\documentclass{article}

\usepackage{arxiv}

\usepackage[utf8]{inputenc} 
\usepackage[T1]{fontenc}    
\usepackage{hyperref}       
\usepackage{url}            
\usepackage{booktabs}       
\usepackage{amsfonts}       
\usepackage{nicefrac}       
\usepackage{microtype}      
\usepackage{lipsum}

\usepackage{morefloats}
\usepackage{color}
\usepackage{multirow}
\usepackage{latexsym}
\usepackage{amsmath,amssymb,amsthm,graphicx}
\usepackage{dsfont}
\usepackage{enumitem}
\usepackage{hyperref}
\usepackage{arydshln}
\usepackage{lscape}
\newtheoremstyle{thm}
{9pt}
{9pt}
{\itshape}
{}
{\bfseries}
{.}
{ }
{}
\theoremstyle{thm}

\newtheorem{theorem}{Theorem}[section]
\newtheorem{lemma}[theorem]{Lemma}
\newtheorem{corollary}[theorem]{Corollary}

\newtheoremstyle{def}
{9pt}
{9pt}
{}
{}
{\bfseries}
{.}
{ }
{}
\theoremstyle{def}

\newcommand{\R}{\mathbb{R}} 
\newcommand{\N}{\mathbb{N}} 
\newcommand{\E}{\mathbb{E}} 
\newcommand{\PP}{\mathbb{P}} 

\newcommand{\fse}{\overset{a.s.}{\longrightarrow}}
    \def\cd{\stackrel{\mathcal{D}}{\longrightarrow}}
    \def\cp{\stackrel{\mathcal{\PP}}{\longrightarrow}}

\renewcommand{\footnoterule}{%
	\kern -3.5pt
	\hrule width \textwidth height 1pt
	\kern 3.5pt
}

\makeatletter
\def\blfootnote{\xdef\@thefnmark{}\@footnotetext}
\makeatother

\title{On combining the zero bias transform and the empirical characteristic function to test normality}

\author{Bruno Ebner\\
 Institute of Stochastics, \\
Karlsruhe Institute of Technology (KIT), \\
Englerstr. 2, D-76133 Karlsruhe. \\
\texttt{Bruno.Ebner@kit.edu}\\
}

\begin{document}

\date{\today}
\maketitle


\begin{abstract}
We propose a new powerful family of tests of univariate normality. These tests are based on an initial value problem in the space of characteristic functions originating from the fixed point property of the normal distribution in the zero bias transform. Limit distributions of the test statistics are provided under the null hypothesis, as well as under contiguous and fixed alternatives. Using the covariance structure of the limiting Gaussian process from the null distribution, we derive explicit formulas for the first four cumulants of the limiting random element and apply the results by fitting a distribution from the Pearson system. A comparative Monte Carlo power study shows that the new tests are serious competitors to the strongest well established tests.
\end{abstract}

\section{Introduction}
\label{Intro}
In view of the assumption of normality in many classical models, testing for normality is commonly known as the mostly used and discussed goodness-of-fit technique. To be specific, let $X, X_1, X_2, \dotso$ be real-valued independent and identically distributed (iid.) random variables defined on an underlying probability space $(\Omega,\mathcal{A},\mathbb{P})$. The problem of interest is to test the hypothesis
\begin{equation}\label{H0}
	H_0: \mathbb{P}^X \in \mathcal{N}=\{ N(\mu, \sigma^2) \, | \, (\mu, \sigma^2) \in \R \times (0,\infty) \}
\end{equation}
against general alternatives. This testing problem has been considered extensively and a multitude of different test statistics is available. The classical tests are based on the empirical distribution function, like the Kolmogorov-Smirnov test (modified in \cite{L:1967}), the Anderson-Darling test, see \cite{AD:1952}, the empirical characteristic function, see \cite{EP:1983}, the empirical moment generating function, see \cite{HV:2019}, on empirical measures of skewness and kurtosis, see \cite{BS:1975, JB:1980, PDB:1977} (known to lead to inconsistent procedures), the Wasserstein distance, see \cite{Betal:2000}, measures of entropy, see \cite{TAA:2019,V:1976}, the integrated empirical distribution function, see \cite{K:2001}, or correlation and regression tests, like the time-honored  ''bench-mark test'' of Shapiro-Wilk , see \cite{SW:1965}, among others. For a survey of classical methods see \cite{Betal:2000}, section 3, and \cite{H:1994}, and for comparative simulation studies, see \cite{BDH:1989,FRS:2006,LL:1992,PDB:1977,RDC:2010,SWC:1968,YS:2011}. For a survey on tests of multivariate normality see \cite{H:2002}, for recent multivariate tests see \cite{DEH:2019}, and for new developments on normality tests for Hilbert space valued random elements, see \cite{HJG:2019,KC:2019}.

Our novel approach relies on the famous Stein characterization and the connected zero bias transform: It is well-known that the normal distribution is the fixed point of the zero bias transform, see \cite{CGS:2011,GR:1997}. Let $X$ be a centred random variable with $\sigma^2=\E(X^2)<\infty$. Following \cite{S:2013}, the characteristic function of the $X$-zero bias transformed random variable $X^*$ is
\begin{equation*}\label{CFZBT}
\E\left(e^{itX^*}\right)=\left\{\begin{array}{cc}\displaystyle -\frac1{\sigma^2}\frac{\varphi^\prime(t)}{t}, & t\in\R\setminus\{0\},\\ 1, & t=0,\end{array}\right.
\end{equation*}
where $\varphi(\cdot)$ is the characteristic function of $X$, and $i$ stands for the imaginary unit. Indeed, (\ref{CFZBT}) represents an operator $A$ mapping from the space of characteristic functions into itself, where $A\varphi=\varphi^\prime/\varphi^{\prime\prime}(0)$, see statement (a) of Theorem 12.2.5 in \cite{L:1970}, and apply $\varphi^{\prime}(0)=i\E(X)=0$ and $\varphi^{\prime\prime}(0)=-\sigma^2$. Together with the assumption $\sigma^2=1$ and the fixed point approach this leads to the initial value problem of an ordinary differential equation
\begin{equation}\label{AWP}
\left\{\begin{array}{c}\varphi^\prime(t)=-t\varphi(t),\\ \varphi(0)=1. \end{array}\right.
\end{equation}
The unique solution of this initial value problem is $\varphi(t)=\exp\left(-t^2/2\right)$, $t\in\R$, which is (confirming the fix point argument) the characteristic function of the standard normal distribution. Note that the moment assumptions ensure the existence of the derivative of the characteristic function, see Corollary 1 and 2 to Theorem 2.3.1 in \cite{L:1970}. To model the standardization assumption leading to (\ref{AWP}), we consider the scaled residuals
\begin{equation*}
Y_{n,j}=\frac{X_j-\overline{X}_n}{S_n},\quad j=1,\ldots,n,
\end{equation*}
where $\overline{X}_n=\frac1n\sum_{j=1}^n X_j$ is the mean and $S_n^2=\frac1{n}\sum_{j=1}^n(X_j-\overline{X}_n)^2$ is the sample variance. Denoting the empirical characteristic function by $\varphi_n(t)=\frac1n\sum_{j=1}^n\exp(itY_{n,j}),\,t\in\R,$ we have $\varphi_n^\prime(t)=\frac1n\sum_{j=1}^niY_{n,j}\exp(itY_{n,j}),\,t\in\R,$ and by estimating both sides of (\ref{AWP}) we propose
the test statistic
\begin{equation*}
Z_{n}=n\int_{-\infty}^{\infty}\left|\frac1n\sum_{j=1}^n\ (iY_{n,j}+t)\exp(itY_{n,j})\right|^2w(t)\mbox{d}t,
\end{equation*}
where $w(\cdot)$ is a suitable bounded weight function and $|x|^2=\mbox{Re}(x)^2+\mbox{Im}(x)^2$ is the squared absolute value of a complex number $x\in\mathbb{C}$. If $X$ originates from a normal distribution, $Z_n$ should be close to zero, and thus rejection of $H_0$ in (\ref{H0}) will be for large values of $Z_n$ (empirical and asymptotic critical values are specified in Section \ref{sec:ApproxCV}). Tacitly, we assume the conditions
\begin{equation}\label{eq:weight}
w(t)=w(-t),\quad t\in\R,\quad \int_{-\infty}^\infty w(t)\mbox{d}t<\infty.
\end{equation}
Note that $Z_n$ depends only on the scaled residuals $Y_{n,1},\ldots, Y_{n,n}$ and is hence invariant under translation or rescaling of the data set $X_1,\ldots,X_n$, which indeed is a desirable property, since the family $\mathcal{N}$ is closed under affine transformations.
Setting $w(t)=w_a(t)=\exp(-at^2)$, $a>0$, a direct evaluation of integrals shows that $Z_n$ takes the form
\begin{eqnarray*}
Z_{n,a}&=&\frac1n\sqrt{\frac{\pi}{a}}\sum_{j,k=1}^n\left(\frac1{4a^2}\left(2a-(Y_{n,j}-Y_{n,k})^2\right)\right.\left.-\frac1{2a}(Y_{n,j}-Y_{n,k})^2+Y_{n,j}Y_{n,k}\right)\exp\left(-\frac1{4a}(Y_{n,j}-Y_{n,k})^2\right),
\end{eqnarray*}
which represents a computational stable and easy to implement version of $Z_n$.
By some expansion of the exponential function and noting that $\sum_{j=1}^nY_{n,j}=0$ and $\sum_{j=1}^nY_{n,j}^2=n$, we have elementwise on the probability space
\begin{equation*}
\lim_{a\rightarrow\infty}\frac{16a^{\frac52}}{3n\sqrt{\pi}}Z_{n,a}=\left(\frac1n\sum_{j=1}^nY_{n,j}^3\right)^2\quad\mbox{and}\quad\lim_{a\rightarrow0}\sqrt{\frac{a}{\pi}}Z_{n,a}-\frac1{2a}=1.
\end{equation*}
It is interesting to see that the limit for $a\rightarrow\infty$ is squared sample skewness, and that this limiting behaviour coincides with the one observed in \cite{HV:2019}, section 4.

The rest of the paper is organized as follows. In Section \ref{sec:AsyNull} we derive the limit distribution of $Z_{n,a}$ under the null hypothesis. Section \ref{sec:ContAlt} states results under a sequence of contiguous alternatives, while in Section \ref{sec:ConsistFixed} we show that the new tests are consistent against alternatives satisfying a weak moment condition. Furthermore, we obtain a central limit result for the test. In Section \ref{sec:ApproxCV}, we derive explicit formulas for the first four cumulants of the limit null distribution of $Z_{n,a}$ and fit the Pearson-system of distributions to approximate the critical values of the test statistic. We complete the paper by a competitive Monte Carlo simulation study in Section \ref{sec:Simu} and finally draw conclusions and identify some open problems for further research in Section \ref{sec:conc}. The paper is concluded by an Appendix that contains proofs and the formula of the fourth cumulant.

\section{Asymptotic null distribution}\label{sec:AsyNull}
A suitable setup for deriving asymptotic theory is the Hilbert space of measurable, square integrable functions $L^2=L^2(\mathbb{R},\mathcal{B}, w\mbox{d}\mathcal{L}^1)$, where $\mathcal{B}$ is the Borel-$\sigma$-field of $\mathbb{R}$ and $\mathcal{L}^1$ is the Lebesgue measure on $\mathbb{R}$. Notice that the functions figuring within the integral in the definition of $Z_{n}$ are $(\mathcal{A} \otimes \mathcal{B}, \mathcal{B})$-measurable random elements of $L^2$. We denote by
\begin{equation*}
	\|f\|_{L^2} = \left( \int_{\R} \big|f(t)\big|^2 \, \omega(t) \, \mathrm{d}t \right)^{1/2}, \qquad \langle f, g \rangle_{L^2}=\int_{\R} f(t)g(t) \, \omega(t) \, \mathrm{d}t
\end{equation*}
the usual norm and inner product in $L^2$.
After straightforward calculations using (\ref{eq:weight}) and symmetry arguments, we have
\begin{equation*}
Z_n=\int_{-\infty}^\infty W_n^2(t)w(t)\mbox{d}t,
\end{equation*}
where
\begin{equation}\label{eq:DefWn}
W_n(t)=\frac1{\sqrt{n}}\sum_{j=1}^n(Y_{n,j}+t)\cos(tY_{n,j})+(t-Y_{n,j})\sin(tY_{n,j}),\quad t\in\R.
\end{equation}
Motivated by a multivariate Taylor expansion we consider the processes
\begin{eqnarray*}
W^*_n(t)&=&\frac1{\sqrt{n}}\sum_{j=1}^n(X_j+t)\cos(tX_j)+(t-X_j)\sin(tX_j)\\
&&+\left((tX_j-(t^2+1))\cos(tX_j)+(tX_j+(t^2+1))\sin(tX_j)\right)\overline{X}_n\\
&&+X_j\left((tX_j-(t^2+1))\cos(tX_j)+(tX_j+(t^2+1))\sin(tX_j)\right)(S_n-1)
\end{eqnarray*}
and
\begin{equation*}
\tilde{W}_n(t)=\frac1{\sqrt{n}}\sum_{j=1}^n(X_j+t)\cos(tX_j)+(t-X_j)\sin(tX_j)-\exp\left(-\frac{t^2}2\right)X_j+t\exp\left(-\frac{t^2}2\right)(X_j^2-1),
\end{equation*}
$t\in\R$. In what follows let $X_1,X_2,\ldots$ be iid. random variables, and in view of affine invariance of $Z_{n}$ we assume w.l.o.g. $X_1\sim N(0,1)$. The following Lemma shows that the processes $W_n$, $W^*_n$ and $\tilde{W}_n$ are asymptotically equivalent. The proof is found in Appendix \ref{app:prooflem1}.
\begin{lemma}\label{lem:conviP}
We have under $H_0$
\begin{equation*}
\|W_n-W^*_n\|_{L^2}\cp0\quad\mbox{and}\quad\|W_n^*-\tilde{W}_n\|_{L^2}\cp0.
\end{equation*}
\end{lemma}
In order to derive the asymptotic null distribution of $Z_n$, it suffices to show the weak convergence of $\tilde{W}_n$ in $L^2$ to a centred Gaussian process.
\begin{theorem}\label{thm:nulldist}
Under the standing assumptions, there is a centred Gaussian random element $W$ of $L^2$ with covariance kernel
\begin{equation*}
K_Z(s,t)=(st+1)\exp\left(-\frac{(s-t)^2}2\right)-(2st+1)\exp\left(-\frac{s^2+t^2}2\right),\quad s,t\in\R,
\end{equation*}
such that with $W_n$ defined in (\ref{eq:DefWn}), we have $W_n\cd W$ in $L^2$ as $n\rightarrow\infty$.
\end{theorem}
\begin{proof}
By Lemma \ref{lem:conviP} it follows that the limit distribution of $W_n$ is the same as that of $\tilde{W}_n$. Note that
\begin{equation*}
\tilde{W}_n(t)=\frac1{\sqrt{n}}\sum_{j=1}^n\widetilde{W}_{n,j}(t),\quad t\in\R,
\end{equation*}
where
\begin{equation*}
\widetilde{W}_{n,j}(t)=(X_j+t)\cos(tX_j)+(t-X_j)\sin(tX_j)-\exp\left(-\frac{t^2}2\right)X_j+t\exp\left(-\frac{t^2}2\right)(X_j^2-1),\quad t\in\R,
\end{equation*}
$j=1,2,\ldots,n$ and $\E\widetilde{W}_{n,1}=0$. Since $\widetilde{W}_{n,1},\widetilde{W}_{n,2},\ldots$ are iid. centred elements of $L^2$, we can directly apply the central limit theorem in Hilbert spaces, see Corollary 10.9 in \cite{LT:1991}. Tedious calculations then show that the covariance kernel $K_Z(s,t)=\E\left(\widetilde{W}_{n,1}(s)\widetilde{W}_{n,1}(t)\right)$ takes the given form.
\end{proof}
The next result follows from a direct application of the continuous mapping theorem.
\begin{corollary}\label{cor:nulldistS}
We have as $n \rightarrow\infty$
\begin{equation*}
Z_n\cd \int_{-\infty}^\infty W^2(t)w(t)\,\mbox{d}t=\|W\|^2_{L^2}.
\end{equation*}
\end{corollary}
We will use this result in Section \ref{sec:ApproxCV} to derive the first four cumulants of the limit random element. As a consequence we obtain approximate critical values by the Pearson system of distributions.

\section{Contiguous alternatives}\label{sec:ContAlt}
In this section we consider a triangular array of row-wise iid. random variables $X_{n,1}, \dots, X_{n,n}$, $n \in \N$, with Lebesgue density
\begin{equation*}
	f_n (t) = f(t) \cdot \left( 1 + \tfrac{c(t)}{\sqrt{n}} \,  \right),\,t\in\R.
\end{equation*}
Here, $f(t) = \tfrac{1}{\sqrt{2 \pi}}\exp(-t^2 /2)$, $t\in\R$, is the density of $N(0,1)$, and $c : \R \to \R$ is a measurable, bounded function satisfying $\int_{-\infty}^\infty c(t) \, f(t) \, \mbox{d}t = 0$. Notice that, since $c$ is bounded, we may assume $n$ to be large enough to ensure $f_n \geq 0$. Setting
\begin{equation*}
\mu_n=\bigotimes_{j=1}^n f\mathcal{L}^1\quad\mbox{and}\quad\nu_n=\bigotimes_{j=1}^n f_n\mathcal{L}^1,
\end{equation*}
it is shown in \cite{BE:2020}, section 4, that by LeCam's first Lemma $\nu_n$ is contiguous to $\mu_n$. Writing
\begin{equation*}
\eta(x,s)=(x+s)\cos(sx)+(s-x)\sin(sx)-\exp\left(-\frac{s^2}2\right)x+s\exp\left(-\frac{s^2}2\right)(x^2-1),\quad x,s\in\R,
\end{equation*}
and following the same lines of proof as in \cite{BE:2020}, section 4, we can show the following result.
\begin{theorem}\label{thm:contalt}
Under the triangular array $X_{n,1}, \dots, X_{n,n}$, we have
\begin{equation*}
Z_n\cd \|W+\zeta\|^2_{L^2},
\end{equation*}
where $W$ is the limiting Gaussian process of Theorem \ref{thm:nulldist} and
\begin{equation*}
\zeta(s)=\int_{-\infty}^{\infty}\eta(x,s)c(x)f(x)\mbox{d}x, \quad s\in\R.
\end{equation*}

\end{theorem}

\section{Consistency and behaviour under fixed alternatives}\label{sec:ConsistFixed}
Let $X,X_1,X_2,\ldots$ be iid. random variables with $\E(X^4)<\infty$. Moreover, we assume $\E(X)=0$ and $\E(X^2)=1$, in view of affine invariance of the test statistic.
\begin{theorem}\label{thm:contalt}
As $n\rightarrow\infty$, we have
\begin{equation*}
\frac{Z_n}{n}\cp \int_{-\infty}^\infty \left|\E\left((iX+t)\exp(itX)\right)\right|^2w(t)\mbox{d}t=\Delta.
\end{equation*}
\end{theorem}
\begin{proof}
Let
\begin{equation*}
W_n^0(t)=\frac1{\sqrt{n}}\sum_{j=1}^n(X_j+t)\cos(tX_j)+(t-X_j)\sin(tX_j),\quad t\in\R.
\end{equation*}
In this setting, we still have $(\overline{X}_n,S_n)\cp(0,1)$ and hence we can apply the same reasoning as in the proof of Lemma \ref{lem:conviP} to show that
\begin{equation*}\label{eq:consis1}
\left\|n^{-1/2}(W_n-W^*_n)\right\|_{L^2}\cp0.
\end{equation*}
Next, we consider (for the definition of $\Psi(\cdot,\cdot)$ see Appendix \ref{app:prooflem1})
\begin{eqnarray*}
n^{-1/2}(W_n^*(t)-W^0_n(t))&=&\frac1n\sum_{j=1}^n\left((tX_j-(t^2+1))\cos(tX_j)+(tX_j+(t^2+1))\sin(tX_j)\right)\overline{X}_n\\
&&+X_j\left((tX_j-(t^2+1))\cos(tX_j)+(tX_j+(t^2+1))\sin(tX_j)\right)(S_n-1)\\
&=&\overline{X}_n\frac1n\sum_{j=1}^n\Psi(t,X_j)+(S_n-1)\frac1n\sum_{j=1}^nX_j\Psi(t,X_j).
\end{eqnarray*}
By the triangle in inequality, we have
\begin{equation}\label{eq:consis2}
\left\|n^{-1/2}(W_n^*-W^0_n)\right\|^2_{L^2}\le2\left|\overline{X}_n\right|^2\left\|\frac1n\sum_{j=1}^n\Psi(\cdot,X_j)\right\|^2_{L^2}+2|S_n-1|^2\left\|\frac1n\sum_{j=1}^nX_j\Psi(\cdot,X_j)\right\|^2_{L^2}.
\end{equation}
By the law of large number in Banach spaces and $(\overline{X}_n,S_n)\cp(0,1)$, the right hand side of (\ref{eq:consis2}) converges to zero in probability. Note that the expectations exist due to the existence of the first two derivatives of the characteristic function of $X$, which is implied by $\E(X^2)<\infty$. Again, by the law of large number in Banach spaces, we have
\begin{equation*}
n^{-1/2}W_n^0(t)\fse \E\left[(X+t)\cos(tX)+(t-X)\sin(tX))\right]
\end{equation*}
in $L^2$. In view of (\ref{eq:consis1}), (\ref{eq:consis2}), and the symmetry of the weight function $w(\cdot)$, some calculations give
\begin{eqnarray*}
\frac{Z_n}{n}=\left\|n^{-1/2}W_n\right\|^2_{L^2}\cp\int_{-\infty}^\infty \left|\E\left[(X+t)\cos(tX)+(t-X)\sin(tX))\right]\right|^2w(t)\mbox{d}t=\Delta.
\end{eqnarray*}
\end{proof}
Notice that, if $g(t)=\E(\exp(itX))$ denotes the characteristic function of $X$, we have $\Delta=0$ if and only if $g=\varphi$, which is shown by the unique solution of the initial value problem (\ref{AWP}). This implies that $Z_n\cp\infty$ for any alternative with existing second moment. Thus we conclude that the test based on $Z_n$ is consistent against each such alternative.

To derive further asymptotic results, we follow the methodology in \cite{BEH:2017}. Put $W^\bullet_n(\cdot)=n^{-1/2}W_n(t)$ and $z(t)=\E\left[(X+t)\cos(tX)+(t-X)\sin(tX))\right]$, we then have
\begin{eqnarray} \nonumber
\sqrt{n}\left( \frac{Z_{n}}{n} - \Delta \right) & = & \sqrt{n} \left( \|W^\bullet_n\|^2_{L^2} - \|z\|^2_{L^2} \right) = \sqrt{n} \langle W^\bullet_n -z,W^\bullet_n+z \rangle_{L^2} \\ \nonumber
& = &  \sqrt{n} \langle W^\bullet_n -z,2z + W^\bullet_n-z \rangle_{L^2}\\ \label{eq:zerlegung}
& = & 2 \langle \sqrt{n}(W^\bullet_n -z),z \rangle_{L^2} + \frac{1}{\sqrt{n}} \| \sqrt{n}\left(W^\bullet_n - z\right)\|^2_{L^2}.
\end{eqnarray}
The following structural Lemma is needed in the subsequent derivations and is proved in Appendix \ref{app:proofconvW}.
\begin{lemma}\label{lem:convW}
We have
\begin{equation*}
\sqrt{n}(W^\bullet_n -z)\cd \mathcal{W},
\end{equation*}
in $L^2$, where $\mathcal{W}$ is a centred Gaussian process in $L^2$ with covariance kernel
\begin{eqnarray*}
K_W(s,t)&=&\E\left[(st+X^2)\cos((s-t)X)+(st-X^2)X\sin((s+t))\right]\\
&&+\E\left[((s-t)\sin((s-t)X)+(s+t)X\cos((s+t)X))\right]\\
&&+\frac12\left[\E\left[(X+s)(b(t)(X^2-1)+2a(t)X)\cos(sX)\right]+\E\left[(b(s)(X^2-1)+2a(s)X)(t+X)\cos(tX)\right]\right.\\
&&\left.+\E\left[(b(t)(X^2-1)+2a(t)X)(s-X)\sin(sX)\right]+\E\left[(b(s)X^2+2a(s)X-b(s))(t-X)\sin(tX)\right]\right]\\
&&+\frac14\left[b(t)\E(X^4)+2a(t)\E(X^3)-b(t)\right]b(s)+\frac12a(s)(b(t)\E(X^3)+2a(t))-z(s)z(t),\quad s,t\in\R,
\end{eqnarray*}
where
\begin{equation*}
a(t)=\E(\Psi(t,X))\quad\mbox{and}\quad b(t)=\E(X\Psi(t,X)).
\end{equation*}
\end{lemma}
Lemma \ref{lem:convW} shows that $\sqrt{n}\left(W^\bullet_n - z\right)$ is a tight sequence in $L^2$, thus we see by Slutzky's Lemma and $\frac{1}{\sqrt{n}} \| \sqrt{n}\left(W^\bullet_n - z\right)\|^2_{L^2}\cp0$ that the limit distribution of $\sqrt{n}\left( \frac{Z_{n}}{n} - \Delta \right)$ in (\ref{eq:zerlegung}) only depends on \linebreak$2 \langle \sqrt{n}(W^\bullet_n -z),z \rangle_{L^2}$. A direct application of Theorem 1 in \cite{BEH:2017} and Lemma \ref{lem:convW} yields the following result.

\begin{theorem}\label{thm:fixedalt}
Under the stated assumptions, we have
\begin{equation*}
\sqrt{n}\left( \frac{Z_{n}}{n} - \Delta \right)\cd N(0,\tau^2),
\end{equation*}
where
\begin{equation*}
\tau^2=4\int_{-\infty}^\infty \int_{-\infty}^\infty K_W(s,t)z(s)z(t)w(s)w(t)\,{\mbox{d}}s\mbox{d}t.
\end{equation*}
\end{theorem}
In principle, one can calculate for a fixed alternative the explicit version of $K_W$ and $z$ and finally $\tau^2$. For most of the alternatives this will be too complicated, thus we suggest to estimate $\tau^2$ by a consistent estimator $\widehat{\tau}_n^2$. Corollary 1 in \cite{BEH:2017} then states that
\begin{equation*}
\frac{\sqrt{n}}{\widehat{\tau}_n}\left( \frac{Z_{n}}{n} - \Delta \right)\cd N(0,1),
\end{equation*}
which opens grounds to applications as suggested in Section 3 in \cite{BEH:2017}, i.e., computation of an asymptotic confidence interval for $\Delta$, approximation of the power function or neighborhood-of-model validation. For specific examples of the needed methodology, see Section 4 of \cite{BEH:2017}, or \cite{BE:2020,DEH:2019} respectively.

\section{Approximation of the limit null distribution}\label{sec:ApproxCV}
In this section we follow the methodology in \cite{H:1990} to approximate the critical values of the asymptotic level $\alpha$ test based on $Z_n$ by exploiting the covariance structure of the limiting centred Gaussian process of Theorem \ref{thm:nulldist}. Let $Z_{\infty,a}=\|W\|_{L^2}$ be the random variable with the limit null distribution of $Z_n$ in Corollary \ref{cor:nulldistS}. Hence $W$ is the random element in Theorem \ref{thm:nulldist} and the weight function $w_a(t)=\exp\left(-at^2\right),$ $t\in\R,$ of Section \ref{Intro} is used. By the results of Corollary \ref{cor:nulldistS} it is well-known that the limiting null distribution of $Z_{n,a}$ is given by the infinite series
\begin{equation*}
Z_{\infty,a}=\sum_{j=1}^\infty \lambda_j(a)Y_j^2.
\end{equation*}
Here, $Y_1,Y_2,\ldots$ being independent $N(0,1)$ distributed random variables and $(\lambda_j(a))_{j\ge1}$ is the decreasing sequence of the positive eigenvalues of the integral operator
\begin{equation*}
\mathcal{K}f(s)=\int_{-\infty}^\infty K_Z(s,t)f(t)w_a(t)\,\mbox{d}t,
\end{equation*}
on $L^2$. Notice that $\mathcal{K}$ depends solely on the covariance kernel $K_Z$ of Theorem \ref{thm:nulldist} and the weight function $w_a(\cdot)$. It seems hopeless to obtain closed-form expressions for the eigenvalues $\lambda_j$, hence we derive the first four moments of $Z_{\infty,a}$ in order to fit the Pearson system of distributions, see \cite{JKB:1994}, chapter 12, section 4.1. The $m$-th cumulant $\kappa_m(a)$ of $Z_{\infty,a}$ is
\begin{equation*}
\kappa_m(a)=2^{m-1}(m-1)!\int_{-\infty}^\infty h_m(t,t)w_a(t)\,\mbox{d}t,
\end{equation*}
where $h_1(s,t)=K_S(s,t)$ and
\begin{equation*}
h_m(s,t)=\int_{-\infty}^\infty h_{m-1}(s,u)K_Z(u,t)w_a(u)\,\mbox{d}u,\quad m\ge2.
\end{equation*}
The formulae for the first three cumulants are (the computations were performed by using the computer algebra system Maple 2019, \cite{Maple2019})
\begin{equation*}
\kappa_1(a)=-1/2\,{\frac {\sqrt {\pi} \left( 2\,{a}^{5/2}-2\,\sqrt {a+1}{a}^{2}+4\,{a}^{3/2}-3\,\sqrt {a+1}a-\sqrt {a+1} \right) }{ \left( a+1 \right)^{3/2}{a}^{3/2}}},
\end{equation*}
\begin{eqnarray*}
\kappa_2(a)&=&\pi\left( a+2 \right) ^{-2} \left( a+3/2 \right) ^{-2} \left( a+1 \right)
^{-3} \left( a+1/2 \right) ^{-2}\cdot\ldots\\
&&\cdot\left(9+66a+{\frac {27}{32\sqrt {a+2}{a}^{5/2}}}+2{
\frac {{a}^{17/2}}{\sqrt {a+2}}}+22{\frac {{a}^{15/2}}{
\sqrt {a+2}}}+{\frac {215{a}^{13/2}}{2\sqrt {a+2}}}+{\frac {
615{a}^{11/2}}{2\sqrt {a+2}}}+{\frac {13045{a}^{3/2}}{
32\sqrt {a+2}}}\right.\\
&&\left.+{\frac {719\sqrt {a}}{4\sqrt {a+2}}}+{
\frac {429}{8\sqrt {a+2}\sqrt {a}}}+{\frac {20729{a}^{5/
2}}{32\sqrt {a+2}}}+{\frac {4575{a}^{9/2}}{8\sqrt {a+2}}}+{
\frac {5817{a}^{7/2}}{8\sqrt {a+2}}}+{\frac {315}{32
\sqrt {a+2}{a}^{3/2}}}\right.\\
&&\left.-{\frac {416a}{\sqrt {4{a}^{2}+8a+3}}}-{\frac {1295{a}^{2}}{\sqrt {4{a}^{2}+8a+3}}}-{
\frac {2377{a}^{3}}{\sqrt {4{a}^{2}+8a+3}}}-{\frac {2835{a
}^{4}}{\sqrt {4{a}^{2}+8a+3}}}-{\frac {2277{a}^{5}}{\sqrt {
4{a}^{2}+8a+3}}}\right.\\
&&\left.+{\frac {799{a}^{2}}{4}}+{\frac {1323
{a}^{3}}{4}}+{\frac {2681{a}^{4}}{8}}+216{a}^{5}+87
{a}^{6}+20{a}^{7}+2{a}^{8}-{\frac {60}{\sqrt {
4{a}^{2}+8a+3}}}\right.\\
&&\left.-{\frac {1230{a}^{6}}{\sqrt {4{a}^{2}+8
a+3}}}-{\frac {430{a}^{7}}{\sqrt {4{a}^{2}+8a+3}}}-{
\frac {88{a}^{8}}{\sqrt {4{a}^{2}+8a+3}}}-{\frac {8{a}^{
9}}{\sqrt {4{a}^{2}+8a+3}}}
\right)
\end{eqnarray*}
and
\begin{eqnarray*}
\kappa_3(a)&=&16\pi^{3/2}(a^2+2a+1/2)^{-3}(a+1/2)^{-2}(a+3/2)^{-3}(4a^2+8a+3)^{-1/2}(2a+3)^{-1/2}\cdot\ldots\\
&&\cdot(2a^2+4a+1)^{-1/2}(a+1)^{-9/2}a^{-7/2}\frac {\sqrt {2{a}^{2}+4a+1}}{4096}\cdot\ldots\\
&&\cdot \left( -5935392\sqrt {2
a+3}\sqrt {4{a}^{2}+8a+3} \left( {a}^
{21/2}+{\frac {58384{a}^{23/2}}{61827}}+{\frac {124090{a}^{{\frac{
25}{2}}}}{185481}}+{\frac {64900}{185481}{a}^{{\frac{27}{2}}}}+{\frac
{24256}{185481}{a}^{{\frac{29}{2}}}}\right.\right.\\
&&\left.+{\frac {6112}{185481}{a}^{{\frac{31}{2}}}}
+{\frac {928}{185481}{a}^{{\frac{33}{2}}}}+{\frac {64}{185481
}{a}^{{\frac{35}{2}}}}+{\frac {3{a}^{7/2}}{41218}}+{\frac {117{a}^
{9/2}}{82436}}+{\frac {1007{a}^{11/2}}{82436}}+{\frac {45691{a}^{
13/2}}{741924}}+{\frac {37834{a}^{15/2}}{185481}}\right.\\
&&\left.+{\frac {175337{a}^{17/2}}{370962}}+{\frac {49256{a}^{19/2}}{61827}} \right) + \sqrt {2a+3}\left( 4320{a}^{7/2}+12288{a}^{{
\frac{37}{2}}}+190464{a}^{{\frac{35}{2}}}+1354752{a}^{{\frac{33}{2
}}}+5872128{a}^{{\frac{31}{2}}}\right.\\
&&\left.\left.+17370624{a}^{{\frac{29}{2}}}+
37211904{a}^{{\frac{27}{2}}}+73348224{a}^{23/2}+69467904{a}^{21/
2}+50733120{a}^{19/2}+28261824{a}^{17/2}\right.\right.\\
&&\left.\left.+11732544{a}^{15/2}+
3487008{a}^{13/2}+694176{a}^{11/2}+82080{a}^{9/2}+59750400{a}^
{{\frac{25}{2}}} \right)\right.\\
&&\left. +4096 \left( a+1/2 \right) ^{
3} \left( a+1 \right) ^{9/2}\sqrt {2a+1} \left( {a}^{2}+2a+1/2
 \right) ^{3} \left( {a}^{5}+4{a}^{4}+11/2{a}^{3}+{\frac {27{a}^{2}}{8}}+{\frac {21a}{16}}+{\frac{15}{32}} \right)  \right)\\
&& -\frac {4772427\sqrt {4{a}^{2}+8a+3}\sqrt {2}\sqrt {2a+3}}{512}
 \left( {a}^{21/2}+{\frac {1627996{a}^{23/2}}{1590809}}+{\frac {
1297344{a}^{{\frac{25}{2}}}}{1590809}}+{\frac {795456}{1590809}{a}^{
{\frac{27}{2}}}}+{\frac {367232}{1590809}{a}^{{\frac{29}{2}}}}\right.\\
&&\left.+{\frac
{11200}{144619}{a}^{{\frac{31}{2}}}}+{\frac {1664}{93577}{a}^{{\frac{
33}{2}}}}+{\frac {3968}{1590809}{a}^{{\frac{35}{2}}}}+{\frac {256}{
1590809}{a}^{{\frac{37}{2}}}}+{\frac {189{a}^{7/2}}{1156952}}+{
\frac {30321{a}^{9/2}}{12726472}}+{\frac {207081{a}^{11/2}}{
12726472}}\right.\\
&&\left.+{\frac {881383{a}^{13/2}}{12726472}}+{\frac {327552{a}^
{15/2}}{1590809}}+{\frac {1442591{a}^{17/2}}{3181618}}+{\frac {
2429741{a}^{19/2}}{3181618}} \right).
\end{eqnarray*}
The formula for $\kappa_4(a)$ can be found in Appendix \ref{app:4cum} from the first four cumulants we can approximate the distribution of $Z_{\infty,a}$ by a member of the Pearson system of distributions, since
\begin{equation*}
\E(Z_{\infty,a})=\kappa_1(a)\quad\mbox{and}\quad \mbox{Var}(Z_{\infty,a})=\kappa_2(a),
\end{equation*}
as well as the parameters of skewness and kurtosis of $Z_{\infty,a}$ are given by
\begin{equation*}
\sqrt{\beta_1}(a)=\frac{\kappa_3(a)}{(\kappa_2(a))^{3/2}}\quad\mbox{and}\quad \beta_2(a)=3+\frac{\kappa_4(a)}{(\kappa_2(a))^2}.
\end{equation*}

\begin{table}[t]
\centering
\begin{tabular}{l|rrrr}

 $a$ & $\E(Z_{\infty,a})$ & Var$(Z_{\infty,a})$ & $\sqrt{\beta_1}(a)$ & $\beta_2(a)$ \\
  \hline
  0.1 & 30.4036 & 304.1938 & 1.4542 & 6.4513 \\
  0.25 & 7.7811 & 31.2928 & 1.7549 & 7.8821 \\
  0.5 & 2.6013 & 4.7153 & 1.9576 & 8.9907 \\
  0.75 & 1.3056 & 1.3821 & 2.0799 & 9.7885 \\
  1 & 0.7787 & 0.5430 & 2.1780 & 10.4822 \\
  3 & 0.0861 & 0.0094 & 2.5812 & 13.3852 \\
  5 & 0.0277 & 0.0011 & 2.7053 & 14.2265 \\
  10 & 0.0055 & 0.0001 & 2.7885 & 14.7597
\end{tabular}
\caption{Values of mean, variance, skewness and kurtosis of $Z_{\infty,a}$ rounded to 4 decimals}\label{tab:MVSK}
\end{table}
These values can directly be used in packages that implement the Pearson system, for concrete values see Table \ref{tab:MVSK}. In the statistical computing language \texttt{R}, see \cite{R:2019}, we will use the package \texttt{PearsonDS}, see \cite{BK:2017}, to approximate critical values, see Table \ref{tab:cv}, and $p$-values of the corresponding tests.

\section{Simulations}\label{sec:Simu}
This section presents results of a comparative finite sample power simulation study. The study is designed to match and complement  the counterparts in \cite{DEH:2019},
Section 7, and in \cite{BE:2020}, Section 6, since we take exactly the same setting with regard to sample size, nominal level of significance
and selected alternative distributions. In this way, we facilitate the comparison with existing procedures, even with some procedures not covered here.
We consider sample sizes $n\in\{20,50,100\}$ and fix the nominal level of significance throughout all simulations to 0.05. All simulations are performed using the statistical computing environment {\tt R}, see \cite{R:2019}. We simulated empirical critical values under $H_0$ for $Z_{n,a}$ with 100~000 replications, see Table \ref{tab:cv}. The row segment entitled '$Z_{\infty,a}$' gives approximations by the method described in Section \ref{sec:ApproxCV}. Each entry in Table \ref{tab:ER} was simulated with 10~000 replications, and an asterisk '*' denotes a perfect rejection rate of 100\%.

\begin{table}[t]
\centering
\begin{tabular}{c|l|rrrrrr}
 Test & $a\backslash q$& 0.01 & 0.05 & 0.1 & 0.9 & 0.95 & 0.99 \\
  \hline
  & 0.1 & 6.86169 & 10.51119 & 13.04177 & 53.02253 & 62.95695 & 84.99265 \\
  & 0.25 & 1.04896 & 1.85667 & 2.48811 & 15.17277 & 18.74151 & 26.51209 \\
  & 0.5 & 0.21479 & 0.42809 & 0.61130 & 5.43755 & 6.81869 & 9.96010 \\
  $Z_{20,a}$ & 0.75 & 0.07838 & 0.16585 & 0.24721 & 2.75949 & 3.53893 & 5.41328 \\
  & 1 & 0.03616 & 0.08152 & 0.12623 & 1.63633 & 2.14516 & 3.40884 \\
  & 3 & 0.00136 & 0.00467 & 0.00866 & 0.17785 & 0.25076 & 0.44301 \\
  & 5 & 0.00030 & 0.00120 & 0.00225 & 0.05674 & 0.08146 & 0.14904 \\
  & 10 & 0.00004 & 0.00017 & 0.00032 & 0.01116 & 0.01631 & 0.03051 \\
\hline
 &0.1 & 6.44914 & 9.96704 & 12.49225 & 53.05553 & 63.40399 & 86.79261 \\
 &0.25 & 0.99974 & 1.79193 & 2.40350 & 15.05561 & 18.79428 & 26.90751 \\
 &0.5 & 0.20190 & 0.42203 & 0.60621 & 5.40910 & 6.89826 & 10.24374 \\
 $Z_{50,a}$ &0.75 & 0.07554 & 0.16688 & 0.25100 & 2.79329 & 3.59782 & 5.50480 \\
  &1 & 0.03526 & 0.08298 & 0.12909 & 1.69165 & 2.20391 & 3.46662 \\
  &3 & 0.00144 & 0.00464 & 0.00847 & 0.19527 & 0.27076 & 0.46805 \\
  &5 & 0.00030 & 0.00116 & 0.00221 & 0.06399 & 0.09031 & 0.16034 \\
  &10 & 0.00004 & 0.00017 & 0.00033 & 0.01292 & 0.01859 & 0.03359 \\
\hline
 & 0.1 & 6.35269 & 9.91258 & 12.40696 & 53.26156 & 63.85351 & 87.88282 \\
 & 0.25 & 0.99012 & 1.75319 & 2.36455 & 15.13173 & 18.81316 & 27.23384 \\
 & 0.5 & 0.20292 & 0.41909 & 0.60298 & 5.43381 & 6.88596 & 10.38987 \\
 $Z_{100,a}$& 0.75  & 0.07478 & 0.16796 & 0.25341 & 2.80970 & 3.63695 & 5.56701 \\
 & 1 & 0.03583 & 0.08448 & 0.13175 & 1.70456 & 2.23136 & 3.47844 \\
 & 3 & 0.00149 & 0.00478 & 0.00876 & 0.20069 & 0.27509 & 0.45774 \\
 & 5 & 0.00031 & 0.00118 & 0.00228 & 0.06636 & 0.09233 & 0.15819 \\
 & 10 & 0.00004 & 0.00017 & 0.00034 & 0.01354 & 0.01923 & 0.03358 \\
 \hline
 &0.1 & 6.89295 & 9.89596 & 12.27245 & 53.39952 & 63.92766 & 87.89731 \\
 & 0.25 & 1.29920 & 1.83683 & 2.35713 & 15.10009 & 18.73029 & 27.15089 \\
 & 0.5 & 0.32955 & 0.45838 & 0.60581 & 5.41750 & 6.89193 & 10.35260 \\
$Z_{\infty,a}$&  0.75 & 0.13650 & 0.19046 & 0.25773 & 2.81741 & 3.63395 & 5.57065 \\
  & 1 & 0.07292 & 0.10059 & 0.13743 & 1.71902 & 2.23934 & 3.48445 \\
  &3 & 0.00826 & 0.00931 & 0.01142 & 0.20558 & 0.27903 & 0.45910 \\
  &5 & 0.00254 & 0.00274 & 0.00322 & 0.06843 & 0.09441 & 0.15828 \\
  & 10 & 0.00038 & 0.00041 & 0.00048 & 0.01414 & 0.01980 & 0.03371 \\
\end{tabular}
\caption{Empirical quantiles of $Z_{n,a}$ for $n=20,50,100$ (100000 replications) and approximation of the quantiles of $Z_{\infty,a}$ by a Pearson family}\label{tab:cv}
\end{table}

The following alternatives are considered: symmetric distributions, like the Student t$_{\nu}$-distribution with $\nu \in \{3, 5, 10\}$ degrees of freedom, as well as the uniform distribution U$(-\sqrt{3}, \sqrt{3})$, and asymmetric distributions, such as the $\chi^2_{\nu}$-distribution with $\nu \in \{5, 15\}$ degrees of freedom,
the beta distributions B$(1, 4)$ and B$(2, 5)$, and the gamma distributions $\Gamma(1, 5)$ and $\Gamma(5, 1)$, both parametrized by their shape and rate parameter, the Gumbel distribution Gum$(1, 2)$ with location parameter 1
and scale parameter 2, the Weibull distribution W$(1, 0.5)$ with scale parameter 1 and shape parameter 0.5, and the lognormal distribution LN$(0, 1)$.  As representatives of bimodal distributions,  we simulate the mixture of  normal distributions
NMix$(p, \mu, \sigma^2)$, where the random variables are generated by $(1 - p) \, {\rm N}(0, 1) + p \, {\rm N}(\mu, \sigma^2)$, $p \in (0, 1)$, $\mu \in \R$, $\sigma > 0$. Note that these alternatives can also be found in the simulation studies
presented in \cite{BE:2020,DEH:2019a,DEH:2019,RDC:2010}. We chose these alternatives in order to ease the comparison with many other existing tests.

The considered competing test statistics are the following:
\begin{itemize}
\item the Anderson-Darling test, see \cite{AD:1952},
\item the Shapiro-Wilk test, see \cite{SW:1965},
\item the Jarque-Bera test, see \cite{JB:1980},
\item the Henze-Visagie test, see \cite{HV:2019},
\item the Betsch-Ebner test, see \cite{BE:2020},
\item the BHEP test, see \cite{HW:1997},
\item the BCMR test, see \cite{DBCMRR:1999}.
\end{itemize}
Note that these tests are very strong competitors as witnessed by extensive simulation studies, see \cite{RDC:2010}.

We used the implementation of the Anderson-Darling (AD) test in the package \texttt{nortest} from \cite{GL:2015} and the implementation of the Shapiro-Wilk (SW) test from the \texttt{stats} package. The Jarque-Bera (JB) test was implemented in the package \texttt{tseries}, see \cite{TH:2019}. The Henze-Visagie (HV) test uses a weighted $L^2$-type statistic based on a characterization of the moment generating function
that similarly as the newly proposed test employs a first-order differential equation. The univariate statistic is defined by
\begin{equation*}
	{\rm HV}_\gamma= \sqrt{\frac{\pi}{\gamma}} \frac{1}{n}\sum_{j,k = 1}^n \exp\left( \frac{(Y_{n,j} + Y_{n,k})^2}{4 \gamma} \right)\left(Y_{n,j}Y_{n,k}+(Y_{n,j} + Y_{n,k})^2\left(\frac{1}{4 \gamma^2}-\frac{1}{2 \gamma}\right)+\frac{1}{2\gamma}\right),
\end{equation*}
where $\gamma>2$. In what follows, we consider three different tuning parameters $\gamma\in\{2.5,5,10\}$. Simulated critical values can be found in the arXiv version of \cite{HV:2019}. The Betsch-Ebner (BE) test is based on a $L^2$-distance between the empirical zero-bias transformation and the empirical distribution function. By the same fixed point argument, this distance is minimal under normality. The statistic is given by
\begin{eqnarray*} \label{explicit formula stat1}
	{\rm BE} & = &  \frac{2}{n} \sum\limits_{1 \leq j < k \leq n} \left\{ \vphantom{\exp\left(- \tfrac{Y_{(k)}^2}{2 a}\right)} \left( 1 - \Phi\left( \tfrac{Y_{(k)}}{\sqrt{a}} \right) \right) \left( (Y_{(j)}^2 - 1)(Y_{(k)}^2 - 1) + a Y_{(j)} Y_{(k)} \right)  \right.\\&&\left.+ \frac{a}{\sqrt{2 \pi a}} \, \exp\left(- \tfrac{Y_{(k)}^2}{2 a}\right) \left( - Y_{(j)}^2 Y_{(k)} + Y_{(k)} + Y_{(j)} \right)
	\right\} \\
	& & + \frac{1}{n} \sum\limits_{j=1}^{n} \left\{ \vphantom{\frac{a}{\sqrt{2 \pi a}}} \left( 1 - \Phi\left(\tfrac{Y_{n,j}}{\sqrt{a}}\right) \right) \left( Y_{n,j}^4 + (a - 2) Y_{n,j}^2 + 1 \right)  + \frac{a}{\sqrt{2 \pi a}} \, \exp\left(- \tfrac{Y_{n,j}^2}{2 a}\right) \left( 2 Y_{n,j} - Y_{n,j}^3 \right) \right\},
\end{eqnarray*}
where $Y_{(1)} \leq \dotso \leq Y_{(n)}$ are the order statistics of the scaled residuals $Y_{n,1},\ldots,Y_{n,n}$, and $\Phi(\cdot)$ stands for the distribution function of the standard normal law. The parameter $a>0$ and the corresponding critical values were chosen by the algorithm presented in \cite{BE:2020}.

Tests based on the empirical characteristic function are represented by the Baringhaus-Henze-Epps-Pulley (BHEP) test, see \cite{BH:1988,EP:1983}. The univariate BHEP test with tuning parameter $\beta>0$ uses the test statistic
\begin{eqnarray*}
	{\rm BHEP}
	&=& \frac{1}{n} \sum_{j,k=1}^{n} \exp\left( - \frac{\beta^2}{2} \big(Y_{n,j}-Y_{n,k}\big)^2 \right) - \frac{2}{\sqrt{1+\beta^2}} \sum_{j=1}^{n} \exp\left( - \frac{\beta^2}{2(1+\beta^2)} \, Y_{n,j}^2 \right) + \frac{n}{\sqrt{1+2\beta^2}}.
\end{eqnarray*}
We fix $\beta = 1$ and took the critical values from \cite{H:1990}. Furthermore, we include the quantile correlation test of del Barrio-Cuesta-Albertos-M\'{a}tran-Rodr\'{i}guez-Rodr\'{i}guez (BCMR), based on the $L^2$-Wasserstein distance, see \cite{Betal:2000},  section 3.3, and \cite{DBCMRR:1999}. The BCMR statistic is given by
\begin{equation*}
	{\rm BCMR} =
	n \left( 1 - \frac{1}{S_n^2} \left( \sum_{k=1}^n X_{(k)} \int_{\frac{k-1}{n}}^{\frac{k}{n}} \Phi^{-1}(t) \, \mathrm{d}t \right)^2 \right) - \int_{\frac{1}{n+1}}^{\frac {n}{n+1}} \frac{t (1 - t)}{\left( \varphi\left( \Phi^{-1}(t) \right) \right)^2} \, \mathrm{d}t,
\end{equation*}
where $X_{(k)}$ is the $k$-th order statistic of $X_1,\ldots,X_n$, $S_n^2$ is the sample variance, and $\Phi^{-1}(\cdot)$ is the quantile function of the standard normal distribution. Simulated critical values can be found in \cite{K:2009}.

\begin{table}[t]
\small
	\setlength{\tabcolsep}{.8mm}
\centering
\begin{tabular}{cl|cccccccc|ccc|cccccc}
  & & \multicolumn{8}{c}{$Z_{n,a}$} & \multicolumn{3}{c}{HV$_\gamma$} & & & & & & \\
  Alt. & $n\backslash a$ & 0.1 & 0.25 & 0.5 & 0.75 & 1 & 3 & 5 & 10 & 2.5 & 5 & 10 & \rm{SW} & \rm{ BCMR} & \rm{ BHEP} & \rm{ AD} & \rm{JB} & \rm{BE}\\
  \hline
\multirow{3}{*}{N$(0,1)$} & 20 & 5 & 5 & 4 & 4 & 4 & 5 & 5 & 5 & 5 & 5 & 5 & 5 & 5 & 5 & 5 & 2 & 5\\
  &50 & 5 & 5 & 5 & 5 & 5 & 5 & 5 & 5 & 5 & 5 & 5 & 5 & 5 & 5 & 5 & 4 & 5\\
  &100 & 5 & 5 & 5 & 5 & 5 & 5 & 5 & 5 & 5 & 5 & 5 & 5 & 5 & 5 & 5 & 4 & 5\\
  \hline
 \multirow{3}{*}{NMix$(0.3,1,0.25)$} & 20 & 25 & 28 & 27 & 24 & 23 & 18 & 17 & 17 & 11 & 13 & 15 & 29 & 29 & 29 & 30 & 7 & 25\\
 & 50 & 61 & 65 & 62 & 59 & 57 & 45 & 41 & 38 & 15 & 24 & 31 & 60 & 60 & 62 & 68 & 19 & 56 \\
 & 100 & 92 & 93 & 91 & 89 & 86 & 76 & 72 & 68 & 23 & 45 & 56 & 88 & 88 & 90 & 94 & 50 & 86 \\
 \multirow{3}{*}{NMix$(0.5,1,4)$} & 20 & 35 & 42 & 42 & 41 & 40 & 33 & 32 & 32 & 32 & 32 & 31 & 39 & 41 & 41 & 46 & 25 & 34\\
 & 50 & 78 & 85 & 83 & 80 & 77 & 59 & 53 & 48 & 49 & 50 & 46 & 77 & 78 & 79 & 86 & 59 & 52 \\
 & 100 & 98 & 99 & 99 & 98 & 97 & 86 & 77 & 66 & 68 & 68 & 63 & 98 & 98 & 98 & 99 & 87 & 75\\
 \hline
  \multirow{3}{*}{$t_3$} & 20 & 19 & 27 & 33 & 35 & 36 & 36 & 36 & 36 & 39 & 38 & 37 & 35 & 37 & 34 & 34 & 32 & 30\\
 & 50 & 37 & 52 & 61 & 64 & 65 & 62 & 61 & 58 & 66 & 64 & 60 & 64 & 66 & 62 & 61 & 67 & 41\\
 & 100 & 63 & 79 & 85 & 87 & 87 & 84 & 82 & 78 & 85 & 84 & 80 & 87 & 88 & 86 & 85 & 89 & 54\\
 \multirow{3}{*}{$t_5$} & 20 & 8 & 12 & 16 & 18 & 19 & 20 & 20 & 21 & 22 & 22 & 22 & 19 & 20 & 18 & 17 & 17 & 16\\
 & 50 & 14 & 22 & 30 & 34 & 35 & 36 & 35 & 34 & 40 & 39 & 36 & 36 & 38 & 31 & 30 & 39 & 22\\
 & 100 & 23 & 38 & 48 & 52 & 54 & 53 & 50 & 47 & 59 & 57 & 51 & 56 & 58 & 50 & 48 & 63 & 27\\
 \multirow{3}{*}{$t_{10}$} & 20 & 6 & 6 & 8 & 9 & 10 & 11 & 11 & 11 & 12 & 12 & 11 & 10 & 10 & 9 & 9 & 8 & 9\\
 & 50 & 6 & 8 & 11 & 13 & 14 & 16 & 16 & 16 & 20 & 19 & 18 & 15 & 17 & 12 & 11 & 18 & 11\\
 & 100 & 8 & 12 & 16 & 19 & 20 & 23 & 22 & 21 & 29 & 27 & 24 & 24 & 25 & 18 & 16 & 29 & 11\\
 \hline
  \multirow{3}{*}{U$(-\sqrt{3},\sqrt{3})$} & 20 & 18 & 21 & 15 & 8 & 4 & 1 & 1 & 1 & 0 & 0 & 0 & 20 & 17 & 13 & 17 & 0 & 4\\
 & 50 & 47 & 59 & 59 & 50 & 36 & 2 & 1 & 0 & 0 & 0 & 0 & 75 & 69 & 54 & 57 & 0 & 3\\
 & 100 & 87 & 94 & 95 & 94 & 91 & 8 & 1 & 0 & 0 & 0 & 0 & * & 99 & 94 & 95 & 57 & 5\\
 \hline
  \multirow{3}{*}{$\chi^2_5$} & 20 & 18 & 28 & 36 & 40 & 41 & 40 & 39 & 40 & 32 & 35 & 38 & 43 & 43 & 42 & 38 & 24 & 44\\
 & 50 & 44 & 66 & 78 & 82 & 84 & 85 & 84 & 83 & 62 & 74 & 79 & 89 & 88 & 83 & 80 & 68 & 84\\
 & 100 & 80 & 95 & 98 & 99 & 99 & 99 & 99 & 99 & 89 & 97 & 98 & * & * & 99 & 99 & 97 & 99\\
 \multirow{3}{*}{$\chi^2_{15}$} & 20 & 7 & 11 & 15 & 16 & 17 & 18 & 18 & 18 & 16 & 17 & 18 & 18 & 18 & 17 & 15 & 11 & 18\\
 & 50 & 13 & 23 & 33 & 39 & 42 & 45 & 45 & 45 & 31 & 37 & 42 & 43 & 43 & 40 & 35 & 31 & 44\\
 & 100 & 21 & 42 & 59 & 67 & 71 & 76 & 77 & 76 & 50 & 65 & 72 & 74 & 74 & 68 & 61 & 60 & 74\\
 \hline
  \multirow{3}{*}{B$(1,4)$} & 20 & 32 & 41 & 47 & 48 & 48 & 43 & 41 & 41 & 27 & 34 & 38 & 59 & 58 & 52 & 51 & 20 & 49\\
 & 50 & 78 & 88 & 92 & 93 & 94 & 91 & 89 & 87 & 51 & 73 & 81 & 99 & 98 & 94 & 95 & 67 & 90\\
 & 100 & 99 & * & * & * & * & * & * & * & 84 & 98 & 99 & * & * & * & * & 99 & *\\
 \multirow{3}{*}{B$(2,5)$} & 20 & 9 & 12 & 14 & 15 & 14 & 13 & 13 & 13 & 9 & 11 & 12 & 17 & 17 & 17 & 15 & 5 & 15\\
 & 50 & 18 & 31 & 39 & 42 & 43 & 39 & 37 & 36 & 14 & 22 & 29 & 50 & 48 & 44 & 39 & 15 & 40\\
 & 100 & 38 & 63 & 75 & 79 & 81 & 79 & 76 & 73 & 23 & 51 & 64 & 90 & 89 & 80 & 76 & 51 & 73\\
 \hline
  \multirow{3}{*}{$\Gamma(1,5)$} & 20 & 55 & 66 & 73 & 74 & 75 & 71 & 70 & 69 & 54 & 62 & 66 & 83 & 82 & 77 & 77 & 47 & 76\\
 & 50 & 96 & 99 & 99 & 99 & * & 99 & 99 & 99 & 90 & 96 & 98 & * & * & * & * & 96 & 99\\
 & 100 & * & * & * & * & * & * & * & * & * & * & * & * & * & * & * & * & *\\
 \multirow{3}{*}{$\Gamma(5,1)$} & 20 & 9 & 14 & 19 & 22 & 23 & 23 & 24 & 24 & 20 & 22 & 23 & 23 & 24 & 23 & 20 & 14 & 25\\
 & 50 & 18 & 34 & 47 & 53 & 56 & 59 & 59 & 59 & 40 & 49 & 54 & 59 & 58 & 54 & 48 & 42 & 58\\
 & 100 & 35 & 63 & 79 & 84 & 87 & 90 & 90 & 89 & 65 & 81 & 86 & 90 & 90 & 85 & 81 & 78 & 88\\
 \hline
  \multirow{3}{*}{W$(1,0.5)$} & 20 & 56 & 68 & 74 & 75 & 76 & 72 & 71 & 70 & 56 & 63 & 67 & 84 & 83 & 78 & 78 & 49 & 76\\
 & 50 & 96 & 99 & 99 & * & * & 99 & 99 & 99 & 90 & 97 & 98 & * & * & * & * & 96 & 99\\
 & 100 & * & * & * & * & * & * & * & * & * & * & * & * & * & * & * & * & *\\
 \multirow{3}{*}{Gum$(1,2)$} & 20 & 12 & 19 & 26 & 29 & 31 & 31 & 32 & 32 & 27 & 29 & 31 & 31 & 31 & 31 & 27 & 20 & 32\\
 & 50 & 24 & 44 & 58 & 65 & 68 & 71 & 70 & 70 & 53 & 62 & 67 & 68 & 69 & 65 & 60 & 55 & 70\\
 & 100 & 47 & 76 & 87 & 91 & 93 & 95 & 95 & 95 & 80 & 90 & 93 & 94 & 94 & 92 & 89 & 89 & 94\\
 \multirow{3}{*}{LN$(0,1)$} & 20 & 76 & 84 & 88 & 90 & 90 & 88 & 87 & 87 & 77 & 82 & 85 & 93 & 93 & 91 & 90 & 72 & 90\\
 & 50 & 99 & * & * & * & * & * & * & * & 99 & * & * & * & * & * & * & * & *\\
 & 100 & * & * & * & * & * & * & * & * & * & * & * & * & * & * & * & * & *
\end{tabular}
\caption{Empirical rejection rates of $Z_{n,a}$ and competitors ($\alpha=0.05$, 10~000 replications).}\label{tab:ER}
\end{table}

The results presented in Table \ref{tab:ER} show that the power of $Z_{n,a}$ depends on the choice of the tuning parameter $a>0$. In most cases one is able to find a value of $a$ in which the tests are nearly as good or better than the competitors. Note that for higher values of $a$ $Z_{n,a}$ performs best for the $\chi^2_{15}$, the $\Gamma(5,1)$ and the Gum$(1,2)$ distribution. Very interesting is the behaviour of the HV-test for the uniform U$(-\sqrt{3},\sqrt{3})$, where it fails to detect the alternative for any value of $\gamma$ for any $n$. Another interesting comparison can be made for this uniform distribution between $Z_{n,a}$ and the BE-test if one also takes Table 3 of \cite{BE:2020} into consideration, since it seems that even though both procedures are based on the zero bias transform, $Z_{n,a}$ seems to attain higher power for some values of $a$, while the BE-test seems to be much less sensitive to the actual choice of $a$. The AD-test performs best for the normal mixture distributions, while the overall the SW-test has a strong power for most asymmetric distributions.

Depending of the nature of the alternatives, we would suggest to use $a=0.25$ for symmetric alternatives and $a=3$ for asymmetric alternatives for performing the test. If nothing is known about the nature of the alternative, we suggest to use $a=1$, as it seems to have a good overall power performance. Naturally, it would be interesting to implement a data driven choice of the tuning parameter as suggested by \cite{AS:2015} and corrected in \cite{T:2019}, but we leave this pronlem open for further research.

\section{Conclusions}\label{sec:conc}
We have proposed a new family of tests for normality based on the fixed point property of the zero bias transformation and its corresponding characteristic function. These tests are universally consistent under weak moment conditions and show a remarkable power performance in comparison to the strongest time-honored tests of normality. Weak convergence results under the null hypothesis, under contiguous and fixed alternatives were derived, which open ground to further insights on the behaviour of the tests.

Finally, we point out some open problems concerning the test statistic. A first step to further investigation, would be to derive a consistent estimator of the limiting variance $\tau^2$ in Theorem \ref{thm:fixedalt}. The approximation of the eigenvalues connected to the limiting random element $Z_{\infty,a}$ would give some further insight to approximate Bahadur efficiency statements and the structure of the initial value problem gives hope to extend the procedure to the multivariate case.

\section*{Acknoledgement}
We thank Norbert Henze for numerous suggestions that led to an improvement of the paper.

\bibliographystyle{abbrv}
\bibliography{lit_ZBCF_univ}  

\begin{appendix}
\section{Proofs}
\subsection{Proof of Lemma \ref{lem:conviP}}\label{app:prooflem1}
\begin{proof}
Let for $t,x\in\R$
\begin{equation*}
\Psi_{t,x}(\mu,\sigma)=\frac{x-\mu}{\sigma}\left(\cos\left(t\frac{x-\mu}{\sigma}\right)-\sin\left(t\frac{x-\mu}{\sigma}\right)\right)+t\left(\cos\left(t\frac{x-\mu}{\sigma}\right)+\sin\left(t\frac{x-\mu}{\sigma}\right)\right),
\end{equation*}
and notice that a first order multivariate Taylor approximation around $(\mu_0,\sigma_0)=(0,1)$ gives
\begin{equation*}
\Psi_{t,x}(\mu,\sigma)=\Psi_{t,x}(0,1)+\frac{\partial\Psi_{t,x}(0,1)}{\partial \mu}\mu+\frac{\partial\Psi_{t,x}(0,1)}{\partial \sigma}(\sigma-1),
\end{equation*}
and note that higher order terms involve terms of higher power of $\mu$ and $\sigma$. With that notation and $(\overline{X}_n,S_n)\cp(0,1)$ we have
\begin{equation*}
W_n(t)=\frac1{\sqrt{n}}\sum_{j=1}^n\Psi_{t,X_j}(\overline{X}_n,S_n)=W^*_n(t)+o_{\mathbb{P}}(1),\quad t\in\R,
\end{equation*}
and hence $\|W_n-W^*_n\|_{L^2}=o_{\mathbb{P}}(1)$ by application of the triangle inequality, Slutzky's Lemma and the central limit theorem in $L^2$ implying the boundedness in probability of mixed remainder terms.
Next, note that for $X_1\sim N(0,1)$ and $t\in \R$, we have by symmetry
\begin{equation*}
\E(\sin(tX_1))=\E(X_1\cos(tX_1))=\E(X_1^2\sin(tX_1))=0
\end{equation*}
and using the standard normal characteristic function $\varphi(t)=\E(\cos(tX_1))=\exp(-t^2/2)$ and its derivatives, we have
\begin{equation*}
\E(X_1\sin(tX_1))=t\exp(-t^2/2),\quad\mbox{and}\quad \E(X_1^2\cos(tX_1))=(1-t^2)\exp(-t^2/2).
\end{equation*}
Let
\begin{equation*}
\Psi(t,x)=(tx-(t^2+1))\cos(tx)+(tx+(t^2+1))\sin(tx),
\end{equation*}

Hence, we have for $t\in\R$
\begin{eqnarray*}
\E\left(\Psi(t,X_1)\right)&=&-\exp(-t^2/2)=-\varphi(t),\\
\E\left(X_1\Psi(t,X_1))\right)&=&2t\exp(-t^2/2)=-2\varphi'(t).
\end{eqnarray*}
Now, it is easy to see that
\begin{equation*}
\left\|\frac1n\sum_{j=1}^n\Psi(\cdot,X_j)+\varphi(\cdot)\right\|_{L^2}=o_{\mathbb{P}}(1)\quad\mbox{and}\quad\left\|\frac1n\sum_{j=1}^nX_j\Psi(\cdot,X_j)+2\varphi'(\cdot)\right\|_{L^2}=o_{\mathbb{P}}(1).
\end{equation*}
Since
\begin{equation*}
\sqrt{n} \, (S_n-1) = \frac{1}{\sqrt{n}} \sum\limits_{j=1}^{n} \frac{1}{2} \left(X_j^2-1 \right) + o_\mathbb{P}(1) ,
\end{equation*}
we have
\begin{eqnarray*}
&&\|W_n^*-\tilde{W}_n\|^2_{L^2}\\
&=&\int_{-\infty}^\infty\left|\frac1{\sqrt{n}}\sum_{j=1}^n\Psi(t,X_1)\overline{X}_n+X_j\Psi(t,X_j)(S_n-1)+\varphi(t)\overline{X}_n+\varphi'(t)(X_j^2-1)\right|^2w(t)\mbox{d}t\\
&\le&\int_{-\infty}^\infty\left|\left(\frac1{n}\sum_{j=1}^n\Psi(t,X_1)+\varphi(t)\right)\frac1{\sqrt{n}}\sum_{l=1}^nX_l+\left(\frac1{n}\sum_{j=1}^nX_j\Psi(t,X_j)+\varphi'(t)\right)\frac1{\sqrt{n}}\sum_{l=1}^n(X_l^2-1)\right|^2w(t)\mbox{d}t+o_\mathbb{P}(1)\\
&\le&2\left\{\left\|\frac1n\sum_{j=1}^n\Psi(\cdot,X_j)+\varphi(\cdot)\right\|_{L^2}^2\left(\frac1{\sqrt{n}}\sum_{l=1}^nX_l\right)^2+\left\|\frac1n\sum_{j=1}^nX_j\Psi(\cdot,X_j)+2\varphi'(\cdot)\right\|_{L^2}^2\left(\frac1{\sqrt{n}}\sum_{l=1}^n(X_l^2-1)\right)^2\right\}+o_\mathbb{P}(1)
\end{eqnarray*}
and since $n^{-1/2}\sum_{l=1}^nX_l$ and $n^{-1/2}\sum_{l=1}^n(X_l^2-1)$ are tight sequences, the result follows by Slutsky's Lemma.
\end{proof}
\subsection{Proof of Lemma \ref{lem:convW}}\label{app:proofconvW}
\begin{proof}
Set
\begin{equation*}
\tilde{W}_n^\bullet(t)=\frac{1}{\sqrt{n}}\sum_{j=1}^n(X_j+t)\cos(tX_j)+(t-X_j)\sin(tX_j)+\E(\Psi(t,X))X_j+\frac12\E(X\Psi(t,X))(X_j^2-1)
\end{equation*}
by the same arguments as in the proof of Lemma \ref{lem:conviP}, we have $\|W_n-\tilde{W}_n^\bullet\|_{L^2}\cp0$. Now, by the central limit theorem in Hilbert spaces, we have
\begin{equation*}
\sqrt{n}\left(\frac{\tilde{W}_n^\bullet}{\sqrt{n}}-z\right)\cd \mathcal{W},
\end{equation*}
where $\mathcal{W}\in L^2$ is a centered Gaussian process with covariance kernel $K_W(s,t)=\mbox{Cov}(\tilde{W}_1^\bullet(s),\tilde{W}_1^\bullet(t))$. The stated formula is derived by straightforward calculation.
\end{proof}

\section{4th cumulant of $Z_{\infty,a}$}\label{app:4cum}
\begin{landscape}
{\tiny
\begin{eqnarray*}
\kappa_4(a)&=&48\,{\pi}^{2} \left( {a}^{2}+5/2\,a+5/4 \right) ^{-5} \left( a+1/2 \right) ^{-5}
 \left( a+2 \right) ^{-4} \left( {a}^{2}+2\,a+1/2 \right) ^{-3} \left( {
a}^{2}+3/2\,a+1/4 \right) ^{-5} \left( a+1 \right) ^{-6} \left( a+3/2
 \right) ^{-5}(2\,{a}^{3}+6\,{a}^{2}+5\,a+1)^{-1/2}({a}^{2}+3\,a+2
)^{-1/2}{a}^{-9/2}\cdot\ldots\\&&
\cdot
\left( -\frac {15888388972237\,\sqrt {4\,{a}^{3}+12\,{a
}^{2}+11\,a+3}\sqrt {4\,{a}^{2}+8\,a+3}\sqrt {2\,{a}^{3}+6\,{a}^{2}+5
\,a+1}\sqrt {16\,{a}^{4}+64\,{a}^{3}+84\,{a}^{2}+40\,a+5}}{67108864}
 \left( \sqrt {a+1} \left( {\frac {10016986656268288}{15888388972237}{
a}^{{\frac{67}{2}}}}+{\frac {1444270220426605\,{a}^{{\frac{25}{2}}}}{
63553555888948}}\right.\right.\right.\\
&&\left.\left.\left.+{\frac {82389179722704\,{a}^{23/2}}{15888388972237}}+
{a}^{21/2}+{\frac {10148404334297\,{a}^{19/2}}{63553555888948}}+{
\frac {1304640003889\,{a}^{17/2}}{63553555888948}}+{\frac {32404982280
\,{a}^{15/2}}{15888388972237}}+{\frac {259076456429254198}{
15888388972237}{a}^{{\frac{39}{2}}}}+{\frac {144535729392983458}{
15888388972237}{a}^{{\frac{37}{2}}}}+{\frac {243542217879163392}{
15888388972237}{a}^{{\frac{59}{2}}}}
\right.\right.\right.\\
&&\left.\left.\left.+{\frac {400118395950180864}{
15888388972237}{a}^{{\frac{57}{2}}}}+{\frac {586996844600676480}{
15888388972237}{a}^{{\frac{55}{2}}}}+{\frac {770871902940374400}{
15888388972237}{a}^{{\frac{53}{2}}}}+{\frac {21115219451189760}{
369497417959}{a}^{{\frac{51}{2}}}}+{\frac {960527932076785920}{
15888388972237}{a}^{{\frac{49}{2}}}}+{\frac {7501512704}{
15888388972237}{a}^{{\frac{83}{2}}}}+{\frac {100514398208}{
15888388972237}{a}^{{\frac{81}{2}}}}\right.\right.\right.\\
&&\left.\left.\left.+{\frac {975356035072}{
15888388972237}{a}^{{\frac{79}{2}}}}+{\frac {7304351055872}{
15888388972237}{a}^{{\frac{77}{2}}}}+{\frac {43933586817024}{
15888388972237}{a}^{{\frac{75}{2}}}}+{\frac {72157365052385701}{
15888388972237}{a}^{{\frac{35}{2}}}}+{\frac {32130704968678825}{
15888388972237}{a}^{{\frac{33}{2}}}}+{\frac {25411446550103755}{
31776777944474}{a}^{{\frac{31}{2}}}}+{\frac {8873302935379397}{
31776777944474}{a}^{{\frac{29}{2}}}}\right.\right.\right.\\
&&\left.\left.\left.+{\frac {5432365877003493}{
63553555888948}{a}^{{\frac{27}{2}}}}+{\frac {2418660\,{a}^{9/2}}{
15888388972237}}+{\frac {26902928333987840}{15888388972237}{a}^{{\frac
{65}{2}}}}+{\frac {63417825270829056}{15888388972237}{a}^{{\frac{63}{2
}}}}+{\frac {131966188294238208}{15888388972237}{a}^{{\frac{61}{2}}}}+
{\frac {2332420110\,{a}^{13/2}}{15888388972237}}+{\frac {108093420\,{a
}^{11/2}}{15888388972237}}\right.\right.\right.\\
&&\left.\left.\left.+{\frac {601727316318085656}{15888388972237}
{a}^{{\frac{43}{2}}}}+{\frac {8388608}{369497417959}{a}^{{\frac{85}{2}
}}}+{\frac {913594901408998032}{15888388972237}{a}^{{\frac{47}{2}}}}+{
\frac {781693114431576816}{15888388972237}{a}^{{\frac{45}{2}}}}+{
\frac {416552354601936792}{15888388972237}{a}^{{\frac{41}{2}}}}+{
\frac {218077726638080}{15888388972237}{a}^{{\frac{73}{2}}}}+{\frac {
911030547578880}{15888388972237}{a}^{{\frac{71}{2}}}}\right.\right.\right.\\
&&\left.\left.\left.+{\frac {
3250400705183744}{15888388972237}{a}^{{\frac{69}{2}}}}+{\frac {8388608
}{15888388972237}{a}^{{\frac{87}{2}}}} \right) \sqrt {{a}^{2}+3\,a+2}+
\frac {73701960223245\,\sqrt {a+2}}{63553555888948} \left( {\frac {
49226553320341504}{24567320074415}{a}^{{\frac{67}{2}}}}+{\frac {
1773826939317421\,{a}^{{\frac{25}{2}}}}{73701960223245}}+{\frac {
393110274779764\,{a}^{23/2}}{73701960223245}}+{a}^{21/2}\right.\right.\right.\\
&&\left.\left.\left.+{\frac {
3817681446062\,{a}^{19/2}}{24567320074415}}+{\frac {1434259933009\,{a}
^{17/2}}{73701960223245}}+{\frac {9263307304\,{a}^{15/2}}{
4913464014883}}+{\frac {1614448743288950624}{73701960223245}{a}^{{
\frac{39}{2}}}}+{\frac {866772377781476636}{73701960223245}{a}^{{\frac
{37}{2}}}}+{\frac {858214151772459008}{24567320074415}{a}^{{\frac{59}{
2}}}}+{\frac {1316153654067809792}{24567320074415}{a}^{{\frac{57}{2}}}
}\right.\right.\right.\\
&&\left.\left.\left.+{\frac {362098332677613568}{4913464014883}{a}^{{\frac{55}{2}}}}+{
\frac {447687023824409088}{4913464014883}{a}^{{\frac{53}{2}}}}+{\frac
{498261964927452160}{4913464014883}{a}^{{\frac{51}{2}}}}+{\frac {
2498830444647711936}{24567320074415}{a}^{{\frac{49}{2}}}}+{\frac {
432063643648}{73701960223245}{a}^{{\frac{83}{2}}}}+{\frac {
286898782208}{4913464014883}{a}^{{\frac{81}{2}}}}+{\frac {
11039609454592}{24567320074415}{a}^{{\frac{79}{2}}}}\right.\right.\right.\\
&&\left.\left.\left.+{\frac {
204951751491584}{73701960223245}{a}^{{\frac{77}{2}}}}+{\frac {
1048045253820416}{73701960223245}{a}^{{\frac{75}{2}}}}+{\frac {
417152280084258104}{73701960223245}{a}^{{\frac{35}{2}}}}+{\frac {
11956380864994854}{4913464014883}{a}^{{\frac{33}{2}}}}+{\frac {
22856499656988768}{24567320074415}{a}^{{\frac{31}{2}}}}+{\frac {
23178971747762287}{73701960223245}{a}^{{\frac{29}{2}}}}+{\frac {
6876636097430098}{73701960223245}{a}^{{\frac{27}{2}}}}\right.\right.\right.\\
&&\left.\left.\left.+{\frac {644976
\,{a}^{9/2}}{4913464014883}}+{\frac {361283014419267584}{
73701960223245}{a}^{{\frac{65}{2}}}}+{\frac {260512018086756352}{
24567320074415}{a}^{{\frac{63}{2}}}}+{\frac {100135574979573760}{
4913464014883}{a}^{{\frac{61}{2}}}}+{\frac {650803608\,{a}^{13/2}}{
4913464014883}}+{\frac {29469888\,{a}^{11/2}}{4913464014883}}+{\frac {
1357706227893363264}{24567320074415}{a}^{{\frac{43}{2}}}}+{\frac {
31448891392}{73701960223245}{a}^{{\frac{85}{2}}}}\right.\right.\right.\\
&&\left.\left.\left.+{\frac {
2260384021120766464}{24567320074415}{a}^{{\frac{47}{2}}}}+{\frac {
1844560574332883296}{24567320074415}{a}^{{\frac{45}{2}}}}+{\frac {
540503048824952792}{14740392044649}{a}^{{\frac{41}{2}}}}+{\frac {
903286619373568}{14740392044649}{a}^{{\frac{73}{2}}}}+{\frac {
16645725011050496}{73701960223245}{a}^{{\frac{71}{2}}}}+{\frac {
17689849815269376}{24567320074415}{a}^{{\frac{69}{2}}}}+{\frac {
33554432}{73701960223245}{a}^{{\frac{89}{2}}}}\right.\right.\right.\\
&&\left.\left.\left.+{\frac {1476395008}{
73701960223245}{a}^{{\frac{87}{2}}}} \right)  \right) + \left(
\frac {241982982129125\,\sqrt {a+2}\sqrt {4\,{a}^{2}+8\,a+3}\sqrt {4\,
{a}^{3}+12\,{a}^{2}+11\,a+3}}{4294967296} \left( {\frac {
9170715202355724288}{48396596425825}{a}^{{\frac{67}{2}}}}+{\frac {
824829361875297\,{a}^{{\frac{25}{2}}}}{19358638570330}}+{\frac {
13839308388342\,{a}^{23/2}}{1935863857033}}+{a}^{21/2}\right.\right.\right.\\
&&\left.\left.\left.+{\frac {
221172245710\,{a}^{19/2}}{1935863857033}}+{\frac {20067723385\,{a}^{17
/2}}{1935863857033}}+{\frac {1388405400\,{a}^{15/2}}{1935863857033}}+{
\frac {9097014341614307008}{48396596425825}{a}^{{\frac{39}{2}}}}+{
\frac {4032572610721320364}{48396596425825}{a}^{{\frac{37}{2}}}}+{
\frac {77530781171012009984}{48396596425825}{a}^{{\frac{59}{2}}}}+{
\frac {503148655054716305408}{241982982129125}{a}^{{\frac{57}{2}}}}\right.\right.\right.\\
&&\left.\left.\left.+{
\frac {23501677425911791616}{9679319285165}{a}^{{\frac{55}{2}}}}+{
\frac {617790924040639606784}{241982982129125}{a}^{{\frac{53}{2}}}}+{
\frac {585163120864498180096}{241982982129125}{a}^{{\frac{51}{2}}}}+{
\frac {499256744973609739264}{241982982129125}{a}^{{\frac{49}{2}}}}+{
\frac {1328326413123584}{241982982129125}{a}^{{\frac{83}{2}}}}+{\frac
{8609490215632896}{241982982129125}{a}^{{\frac{81}{2}}}}\right.\right.\right.\\
&&\left.\left.\left.+{\frac {
46128149280325632}{241982982129125}{a}^{{\frac{79}{2}}}}+{\frac {
208355953949540352}{241982982129125}{a}^{{\frac{77}{2}}}}+{\frac {
805192554923425792}{241982982129125}{a}^{{\frac{75}{2}}}}+{\frac {
7935831986201000696}{241982982129125}{a}^{{\frac{35}{2}}}}+{\frac {
552080771887033714}{48396596425825}{a}^{{\frac{33}{2}}}}+{\frac {
168771272531369632}{48396596425825}{a}^{{\frac{31}{2}}}}\right.\right.\right.\\
&&\left.\left.\left.+{\frac {
225198237739750229}{241982982129125}{a}^{{\frac{29}{2}}}}+{\frac {
2081596612710102}{9679319285165}{a}^{{\frac{27}{2}}}}+{\frac {32400\,{
a}^{9/2}}{1935863857033}}+{\frac {740354924891209728}{1935863857033}{a
}^{{\frac{65}{2}}}}+{\frac {166605364226116550656}{241982982129125}{a}
^{{\frac{63}{2}}}}+{\frac {268353384686546747392}{241982982129125}{a}^
{{\frac{61}{2}}}}+{\frac {1073741824}{241982982129125}{a}^{{\frac{93}{
2}}}}\right.\right.\right.\\
&&\left.\left.\left.+{\frac {49392123904}{241982982129125}{a}^{{\frac{91}{2}}}}+{
\frac {68682600\,{a}^{13/2}}{1935863857033}}+{\frac {2160000\,{a}^{11/
2}}{1935863857033}}+{\frac {32883798565740245888}{48396596425825}{a}^{
{\frac{43}{2}}}}+{\frac {32972128387072}{48396596425825}{a}^{{\frac{85
}{2}}}}+{\frac {383500451258021994496}{241982982129125}{a}^{{\frac{47}
{2}}}}+{\frac {264962795071741033024}{241982982129125}{a}^{{\frac{45}{
2}}}}\right.\right.\right.\\
&&\left.\left.\left.+{\frac {91456134448134851632}{241982982129125}{a}^{{\frac{41}{2}
}}}+{\frac {2692695772327051264}{241982982129125}{a}^{{\frac{73}{2}}}}
+{\frac {7862451184013934592}{241982982129125}{a}^{{\frac{71}{2}}}}+{
\frac {20189169394128519168}{241982982129125}{a}^{{\frac{69}{2}}}}+{
\frac {1100316934144}{241982982129125}{a}^{{\frac{89}{2}}}}+{\frac {
15815143325696}{241982982129125}{a}^{{\frac{87}{2}}}} \right)\right.\right.\\
&&\left.\left. -
\frac {386109436024375\,\sqrt {{a}^{2}+3\,a+2}}{536870912} \left( -{
\frac {84304887466191880192}{77221887204875}{a}^{{\frac{67}{2}}}}-{
\frac {1061623519833274\,{a}^{{\frac{25}{2}}}}{15444377440975}}-{
\frac {35806380491592\,{a}^{23/2}}{3088875488195}}-{\frac {
5052981011802\,{a}^{21/2}}{3088875488195}}-{\frac {116854259524\,{a}^{
19/2}}{617775097639}}-{\frac {10758989058\,{a}^{17/2}}{617775097639}}\right.\right.\right.\\
&&\left.\left.\left.-
{\frac {756947760\,{a}^{15/2}}{617775097639}}-{\frac {
130113321129467042688}{386109436024375}{a}^{{\frac{39}{2}}}}-{\frac {
55835065962877271824}{386109436024375}{a}^{{\frac{37}{2}}}}-{\frac {
2196055362565462155264}{386109436024375}{a}^{{\frac{59}{2}}}}-{\frac {
2580351736494166597632}{386109436024375}{a}^{{\frac{57}{2}}}}-{\frac {
392622223567954968576}{55158490860625}{a}^{{\frac{55}{2}}}}\right.\right.\right.
\end{eqnarray*}
\begin{eqnarray*}
&&\left.\left.\left.-{\frac {
2654886604198704375552}{386109436024375}{a}^{{\frac{53}{2}}}}-{\frac {
2326262930245559023104}{386109436024375}{a}^{{\frac{51}{2}}}}-{\frac {
1848456566020192628672}{386109436024375}{a}^{{\frac{49}{2}}}}-{\frac {
65342235671003136}{386109436024375}{a}^{{\frac{83}{2}}}}-{\frac {
313552794399277056}{386109436024375}{a}^{{\frac{81}{2}}}}-{\frac {
7378867027181568}{2206339634425}{a}^{{\frac{79}{2}}}}\right.\right.\right.\\
&&\left.\left.\left.-{\frac {
4616904135234551808}{386109436024375}{a}^{{\frac{77}{2}}}}-{\frac {
2892470402622160896}{77221887204875}{a}^{{\frac{75}{2}}}}-{\frac {
21405868799083127328}{386109436024375}{a}^{{\frac{35}{2}}}}-{\frac {
7296425515279897524}{386109436024375}{a}^{{\frac{33}{2}}}}-{\frac {
2198447109494797632}{386109436024375}{a}^{{\frac{31}{2}}}}-{\frac {
581443725209867406}{386109436024375}{a}^{{\frac{29}{2}}}}\right.\right.\right.\\
&&\left.\left.\left.-{\frac {
5353814148669156}{15444377440975}{a}^{{\frac{27}{2}}}}-{\frac {18720\,
{a}^{9/2}}{617775097639}}-{\frac {21229402741482749952}{11031698172125
}{a}^{{\frac{65}{2}}}}-{\frac {1180180835631042134016}{386109436024375
}{a}^{{\frac{63}{2}}}}-{\frac {1692713927754382053376}{386109436024375
}{a}^{{\frac{61}{2}}}}-{\frac {1073741824}{386109436024375}{a}^{{\frac
{97}{2}}}}-{\frac {51539607552}{386109436024375}{a}^{{\frac{95}{2}}}}\right.\right.\right.\\
&&\left.\left.\left.-
{\frac {171530256384}{55158490860625}{a}^{{\frac{93}{2}}}}-{\frac {
3617973075968}{77221887204875}{a}^{{\frac{91}{2}}}}-{\frac {38138960\,
{a}^{13/2}}{617775097639}}-{\frac {174720\,{a}^{11/2}}{88253585377}}-{
\frac {511484786629896684992}{386109436024375}{a}^{{\frac{43}{2}}}}-{
\frac {460544930742272}{15444377440975}{a}^{{\frac{85}{2}}}}-{\frac {
1331143346975602002432}{386109436024375}{a}^{{\frac{47}{2}}}}\right.\right.\right.\\
&&\left.\left.\left.-{\frac {
867861901016544200352}{386109436024375}{a}^{{\frac{45}{2}}}}-{\frac {
271954859321968227744}{386109436024375}{a}^{{\frac{41}{2}}}}-{\frac {
1599308948797128704}{15444377440975}{a}^{{\frac{73}{2}}}}-{\frac {
19627491529479684096}{77221887204875}{a}^{{\frac{71}{2}}}}-{\frac {
214901751061628780544}{386109436024375}{a}^{{\frac{69}{2}}}}-{\frac {
39628321062912}{77221887204875}{a}^{{\frac{89}{2}}}}\right.\right.\right.\\
&&\left.\left.\left.-{\frac {
1681779270352896}{386109436024375}{a}^{{\frac{87}{2}}}}+ \left( {
\frac {142095030781925720064}{386109436024375}{a}^{{\frac{67}{2}}}}+{
\frac {26208105471852\,{a}^{{\frac{25}{2}}}}{617775097639}}+{\frac {
22013701833097\,{a}^{23/2}}{3088875488195}}+{a}^{21/2}+{\frac {
70913863773\,{a}^{19/2}}{617775097639}}+{\frac {6473344595\,{a}^{17/2}
}{617775097639}}+{\frac {64439680\,{a}^{15/2}}{88253585377}}\right.\right.\right.\right.\\
&&\left.\left.\left.\left.+{\frac {
76842283296113132568}{386109436024375}{a}^{{\frac{39}{2}}}}+{\frac {
33474782549507146248}{386109436024375}{a}^{{\frac{37}{2}}}}+{\frac {
934445814340112560128}{386109436024375}{a}^{{\frac{59}{2}}}}+{\frac {
1151926764144741052416}{386109436024375}{a}^{{\frac{57}{2}}}}+{\frac {
256513948672423369728}{77221887204875}{a}^{{\frac{55}{2}}}}+{\frac {
1290647472597674001408}{386109436024375}{a}^{{\frac{53}{2}}}}\right.\right.\right.\right.\\
&&\left.\left.\left.\left.+{\frac {
1174120309805297830656}{386109436024375}{a}^{{\frac{51}{2}}}}+{\frac {
193088395545537371904}{77221887204875}{a}^{{\frac{49}{2}}}}+{\frac {
9975401333915648}{386109436024375}{a}^{{\frac{83}{2}}}}+{\frac {
11006602781917184}{77221887204875}{a}^{{\frac{81}{2}}}}+{\frac {
256349979600224256}{386109436024375}{a}^{{\frac{79}{2}}}}+{\frac {
146183491321593856}{55158490860625}{a}^{{\frac{77}{2}}}}\right.\right.\right.\right.\\
&&\left.\left.\left.\left.+{\frac {
3540684493395329024}{386109436024375}{a}^{{\frac{75}{2}}}}+{\frac {
12987588438897780244}{386109436024375}{a}^{{\frac{35}{2}}}}+{\frac {
7146557038684388}{617775097639}{a}^{{\frac{33}{2}}}}+{\frac {
1353904109355641646}{386109436024375}{a}^{{\frac{31}{2}}}}+{\frac {
71847624095850058}{77221887204875}{a}^{{\frac{29}{2}}}}+{\frac {
3309922453293424}{15444377440975}{a}^{{\frac{27}{2}}}}+{\frac {10800\,
{a}^{9/2}}{617775097639}}\right.\right.\right.\right.\\
&&\left.\left.\left.\left.+{\frac {38185865328657891328}{55158490860625
}{a}^{{\frac{65}{2}}}}+{\frac {450927587008642531328}{386109436024375}
{a}^{{\frac{63}{2}}}}+{\frac {683913887836570927104}{386109436024375}{
a}^{{\frac{61}{2}}}}+{\frac {1073741824}{386109436024375}{a}^{{\frac{
95}{2}}}}+{\frac {50465865728}{386109436024375}{a}^{{\frac{93}{2}}}}+{
\frac {164282499072}{55158490860625}{a}^{{\frac{91}{2}}}}+{\frac {
22493400\,{a}^{13/2}}{617775097639}}\right.\right.\right.\right.\\
&&\left.\left.\left.\left.+{\frac {713520\,{a}^{11/2}}{
617775097639}}+{\frac {290355193731959736432}{386109436024375}{a}^{{
\frac{43}{2}}}}+{\frac {1496880726933504}{386109436024375}{a}^{{\frac{
85}{2}}}}+{\frac {143425519049003786752}{77221887204875}{a}^{{\frac{47
}{2}}}}+{\frac {480706690484033150464}{386109436024375}{a}^{{\frac{45}
{2}}}}+{\frac {157715599116006339888}{386109436024375}{a}^{{\frac{41}{
2}}}}+{\frac {17145576324857856}{617775097639}{a}^{{\frac{73}{2}}}}\right.\right.\right.\right.\\
&&\left.\left.\left.\left.+{
\frac {28574379833515835392}{386109436024375}{a}^{{\frac{71}{2}}}}+{
\frac {67525511518112448512}{386109436024375}{a}^{{\frac{69}{2}}}}+{
\frac {3385507971072}{77221887204875}{a}^{{\frac{89}{2}}}}+{\frac {
180938852007936}{386109436024375}{a}^{{\frac{87}{2}}}} \right) \sqrt {
4\,{a}^{2}+8\,a+3} \right)  \right) \sqrt {2\,{a}^{3}+6\,{a}^{2}+5\,a
+1}\right.\\
&&\left.+\frac {11799584447520125\,\sqrt {4\,{a}^{2}+8\,a+3}\sqrt {2}
\sqrt {4\,{a}^{3}+12\,{a}^{2}+11\,a+3}\sqrt {a+2}\sqrt {a+1}}{
17179869184} \left( {\frac {2338874529100001705984}{11799584447520125}
{a}^{{\frac{67}{2}}}}+{\frac {17755948088155402\,{a}^{{\frac{25}{2}}}
}{471983377900805}}+{\frac {632454534820939\,{a}^{23/2}}{
94396675580161}}+{a}^{21/2}\right.\right.\\
&&\left.\left.+{\frac {11559938854395\,{a}^{19/2}}{
94396675580161}}+{\frac {1128542725445\,{a}^{17/2}}{94396675580161}}+{
\frac {84261705600\,{a}^{15/2}}{94396675580161}}+{\frac {
308188896351346168064}{2359916889504025}{a}^{{\frac{39}{2}}}}+{\frac {
692768130919254100736}{11799584447520125}{a}^{{\frac{37}{2}}}}+{\frac
{15795271654708880736256}{11799584447520125}{a}^{{\frac{59}{2}}}}\right.\right.\\
&&\left.\left.+{
\frac {19647944888437401321472}{11799584447520125}{a}^{{\frac{57}{2}}}
}+{\frac {4420323669413709512704}{2359916889504025}{a}^{{\frac{55}{2}}
}}+{\frac {22501362691630829371392}{11799584447520125}{a}^{{\frac{53}{
2}}}}+{\frac {20742701122730313433088}{11799584447520125}{a}^{{\frac{
51}{2}}}}+{\frac {3462982744024769019904}{2359916889504025}{a}^{{\frac
{49}{2}}}}+{\frac {160212574372102144}{11799584447520125}{a}^{{\frac{
83}{2}}}}\right.\right.\\
&&\left.\left.+{\frac {177063682063204352}{2359916889504025}{a}^{{\frac{81}
{2}}}}+{\frac {4132075315205767168}{11799584447520125}{a}^{{\frac{79}{
2}}}}+{\frac {3306593531203682304}{2359916889504025}{a}^{{\frac{77}{2}
}}}+{\frac {57363506280532017152}{11799584447520125}{a}^{{\frac{75}{2}
}}}+{\frac {55679266648601820368}{2359916889504025}{a}^{{\frac{35}{2}}
}}+{\frac {99549118197334017808}{11799584447520125}{a}^{{\frac{33}{2}}
}}\right.\right.\\
&&\left.\left.+{\frac {6299684566265892228}{2359916889504025}{a}^{{\frac{31}{2}}}}
+{\frac {8759154059363803308}{11799584447520125}{a}^{{\frac{29}{2}}}}+
{\frac {84918717349357034}{471983377900805}{a}^{{\frac{27}{2}}}}+{
\frac {2494800\,{a}^{9/2}}{94396675580161}}+{\frac {
4423768165953499561984}{11799584447520125}{a}^{{\frac{65}{2}}}}+{
\frac {7508997058376752955392}{11799584447520125}{a}^{{\frac{63}{2}}}}
\right.\right.\\
&&\left.\left.+{\frac {11468876539243907776512}{11799584447520125}{a}^{{\frac{61}{2}
}}}+{\frac {17179869184}{11799584447520125}{a}^{{\frac{95}{2}}}}+{
\frac {807453851648}{11799584447520125}{a}^{{\frac{93}{2}}}}+{\frac {
3680786972672}{2359916889504025}{a}^{{\frac{91}{2}}}}+{\frac {
4508001000\,{a}^{13/2}}{94396675580161}}+{\frac {153522000\,{a}^{11/2}
}{94396675580161}}+{\frac {5521874557177178040576}{11799584447520125}{
a}^{{\frac{43}{2}}}}\right.\right.\\
&&\left.\left.+{\frac {24009421235421184}{11799584447520125}{a}^
{{\frac{85}{2}}}}+{\frac {2616724871911485932544}{2359916889504025}{a}
^{{\frac{47}{2}}}}+{\frac {8942511036426673245184}{11799584447520125}{
a}^{{\frac{45}{2}}}}+{\frac {3075111890856037819648}{11799584447520125
}{a}^{{\frac{41}{2}}}}+{\frac {174166628281696124928}{
11799584447520125}{a}^{{\frac{73}{2}}}}+{\frac {93225084226451275776}{
2359916889504025}{a}^{{\frac{71}{2}}}}\right.\right.\\
&&\left.\left.+{\frac {1106158108427488854016
}{11799584447520125}{a}^{{\frac{69}{2}}}}+{\frac {54206782242816}{
2359916889504025}{a}^{{\frac{89}{2}}}}+{\frac {2899236068786176}{
11799584447520125}{a}^{{\frac{87}{2}}}} \right) + \left( a+1/2
 \right) ^{5} \left( {a}^{12}+12\,{a}^{11}+63\,{a}^{10}+190\,{a}^{9}+{
\frac {5833\,{a}^{8}}{16}}+{\frac {937\,{a}^{7}}{2}}+{\frac {837\,{a}^
{6}}{2}}+272\,{a}^{5}+{\frac {2191\,{a}^{4}}{16}}+{\frac {223\,{a}^{3}
}{4}}\right.\right.\\
&&\left.\left.+{\frac {141\,{a}^{2}}{8}}+{\frac {15\,a}{4}}+{\frac{105}{256}}
 \right)  \left( {a}^{2}+5/2\,a+5/4 \right) ^{5}\sqrt {2\,{a}^{2}+4\,a
+1} \left( {a}^{2}+2\,a+1/2 \right) ^{3} \left( {a}^{2}+3/2\,a+1/4
 \right) ^{5} \left( a+1 \right) ^{2} \left( a+3/2 \right) ^{5}
 \right)
\end{eqnarray*}
}
\end{landscape}
\end{appendix}

\end{document}